\documentclass[leqno,11pt,a4paper]{amsart}

\usepackage{amssymb,latexsym}
\usepackage{graphicx}
\usepackage{hyperref}
\usepackage{amsmath}
\usepackage{amsfonts}
\usepackage{amssymb}
\usepackage{float}
\theoremstyle{plain} 
\newtheorem{theorem}{\bf Theorem}[section]

\newtheorem{lemma}[theorem]{\bf Lemma}

\newtheorem{proposition}[theorem]{\bf Proposition}

\theoremstyle{definition} 
\newtheorem{definition}[theorem]{\bf Definition}
\newtheorem{remark}[theorem]{\bf Remark}

\newtheorem{example}[theorem]{\bf Example}

\newcommand{\q}[1]{``#1''}

\title[Tom \& Jerry Triples with an Application to Fano 3-folds]{Tom \& Jerry Triples with an Application to Fano 3-folds}

\author[Vasiliki~Petrotou]{Vasiliki~Petrotou}
\address{Vasiliki~Petrotou\\ Department of Mathematics \\
University of Ioannina\\
Ioannina, 45110 \\
Greece}
\email{v.petrotou@uoi.gr}

\begin{document}


\begin{abstract}
 \  Unprojection is a theory due to  Reid which constructs more complicated rings starting from  simpler data.
The idea of unprojection is intended for serial use. Papadakis and Neves in \cite{NP2} develop a theory of parallel unprojection.
In the present work we develop a new method of unprojection. Starting from a codimension $3$ ideal defined by the pfaffians of a $5\times 5$ 
skewsymmetric matrix, we use parallel unprojection of Kustin-Miller type in order to construct Gorenstein rings of codimension $6$.
We give two applications. These are two families of  codimension $6$  Fano $3$-folds, in weighted projective 
space which are described in detail in the entries with identifier numbers $14885$ and $12979$ respectively in Brown's Graded Ring Database \cite{BR}.
\end{abstract}

\maketitle

\section{Introduction} 

\label{sec!introduction}

Gorenstein rings play an important role in algebraic geometry. The anticanonical ring of a Fano $n$-fold and the canonical ring of a regular surface of general type are examples of Gorenstein rings which appear often in algebraic geometry.

  If  $R=k[x_1,\dots, x_n]/I$ is  a Gorenstein graded ring which is a quotient of a polynomial ring in $n$ variables divided by a homogeneous ideal $I$ and  the codimension of  $I$, denoted by $\text{codim} \  I$,  is less or equal to $3$ there are good structure theorems.
More precisely,  Serre proved that if  $\text{codim} \  I= 1$ or $2$  then $R$ is  a complete intersection while Buchsbaum and Eisenbud \cite{BE} showed that if  $\text{codim} \  I= 3$ then $I$ is generated by the $2n\times 2n$ pfaffians of a skewsymmetric $(2n+1)\times (2n+1)$ matrix.

An important question is if there exists  a structure theorem for  $\text{codim} \  I \geq 4$.

 In $1983$, Kustin and Miller trying unsuccesfully to answer this question~\cite{KM1, KM2, KM3, KM4, KM5}, introduced a procedure \cite{KM} which constructs more complicated Gorenstein rings from simpler ones by increasing codimension. 
This procedure is called Kustin-Miller unprojection. In $1995$, Reid rediscovered what was essentially the same procedure working with Gorenstein rings arising from K3 surfaces and $3$-folds. Unprojection was studied by Papadakis and Reid in a scheme theoretic formulation \cite{P2, PR}.

Geometrically, unprojection, as indicated by its name, is an inverse of certain projections and can be considered as a modern and explicit version of Castelnuovo contractibility theorem. It has many applications in algebraic geometry. In particular, in the construction of new interesting algebraic surfaces and $3$-folds  \cite{AL, ABR, BG, BKR, BS2, NP1, NP2}. In explicit birational geometry it allows to write down explicitly varieties, morphisms and rational maps that Minimal Model Program says they exist \cite{CM, CPR}. Also, it has some applications in algebraic combinatorics \cite{BP1, BP2, BP3, BP4}.    

Unprojection can be used many times over in an inductive way to produce Gorenstein rings of arbitrary codimension, whose properties are nevertheless, controled by just a few equations as a new unprojection variable is adjoined. Neves and Papadakis \cite{NP2} develop a theory of parallel unprojection.

 The main aim of this paper is to develop a new format of unprojection, which we call Tom and Jerry triples. As a first step, we set conditions in the entries of a $5\times 5$ skewsymmetric matrix $M$ such that the ideal $I$ which is defined by the pfaffians of $M$ to be contained in three codimension $4$ complete intersection ideals $J_1, J_2, J_3$. For this purpose, we use two well-known ways to set up the unprojection data $D\subset Y$ in the case that $Y$ is a codimension~$3$ scheme defined by the pfaffians of a skewsymmetric matrix and $D$ is a codimension~$4$ complete intersection. These are Tom and Jerry unprojections \cite{P1, P2,  R1}. Under the conditions that we set, $M$ can be considered simultaneously as Tom or Jerry matrix in the ideals $J_1, J_2, J_3$. 

As a second step, we fix three codimension~$4$ complete intersection ideals $J_1, J_2, J_3$. Starting from an ideal of codimension~$3$ defined by the pfaffians of a $5\times 5$ skewsymmetric matrix $M$ which is Tom matrix in each of ideals $J_1,J_2,J_3$ we construct,  using parallel Kustin-Miller unprojection,  a Gorenstein ring of codimension equal to $6$. As an application, we construct two families of Fano $3$-folds of codimension~$6$ embedded in weighted projective space which correspond to the entries with ID: $14885$ and ID: $12979$ in Brown's Graded Ring Database~\cite{BR}.

Section~$2$ gives a short overview of the background that we need in this paper. Section~$3$ reviews some existing results related to Kustin-Miller unprojection. Section~$4$ describes a number of alternative ways which guarantee that a codimension $3$ ideal defined by the pfaffians of a $5\times 5$ skewsymmetric matrix is contained in three codimension $4$ complete intersection ideals $J_1, J_2, J_3$. In  Section~$5$ we study in detail one of the cases of Section~$4$. The main result is Theorem \ref{main!thm} which establishes, using the theory of parallel Kustin-Miller unprojection, the construction of a codimension $6$ Gorenstein ring.  We have checked with computer algebra program Macaulay2 \cite{GS} that similar results also hold for the remaining cases studied in Section~$4$.

Section~$6$ contains applications. In Subsections \ref{constr!1} and  \ref{construct!id12979}, we construct two specific $4$-dimensional quotients of the ring studied in Section~$5$ and we prove, partially with use of computer algebra, that the $\text{Proj}$ of each of them corresponds to the families of Fano $3$-folds which we mentioned above.

\section{Notation}
In this section we give some preliminary notions that we use throughout the paper.

\begin{definition}
Assume that $M=[m_ {ij}]$ is a $n\times n$ skewsymmetric matrix,
(   i.e.,  $ m_{ji}=-m_{ij}$  and $m_{ii}=0$ ) with entries in a Noetherian ring $R$.
For $1\leq i\leq n$,  denote by $M_i$ the skewsymmetric submatrix of $M$ obtained by deleting the ith row and ith column of $M$.
\begin{enumerate}
\item If $ n$ is even then $\det M=f(m_ {ij})^2$.
\\
 The polynomial $f(m_ {ij})$ is called the \textit{pfaffian} of the matrix $M$ and is denoted by  $Pf(M)$.
\item  If $ n$ is odd  by \textit{pfaffians} of $M$ we mean the set  $\{ Pf(M_1),Pf(M_2),\dots, Pf(M_n)\}$.
\end{enumerate}
\end{definition}

\begin{example}
\begin{enumerate}
\item For  $n=2$ :
\[
Pf (
\begin{pmatrix} 
0 & m_{12}\\
-m_{12} & 0 
\end{pmatrix}
) = m_{12}.
\]
\item  For  $ n=5$ :
\[
Pf 
 ( \begin{pmatrix} 
0 & m_{12} & m_{13} & m_{14}& m_{15}\\
-m_{12} & 0  & m_{23} & m_{24}& m_{25}\\
-m_{13} & -m_{23}  & 0 & m_{34}&  m_{35}\\
-m_{14} & -m_{24}  & -m_{34} & 0&  m_{45}\\
-m_{15} & -m_{25}  & -m_{35} & -m_{45}&  0
\end{pmatrix}) =  \{ Pf(M_1),Pf(M_2),\dots , Pf(M_5)\}
\]
\end{enumerate}
where,
\[
 Pf(M_1) = m_{23}m_{45} -m_{24}m_{35}+m_{25}m_{34}, 
\]
\[
 Pf(M_2) = m_{13}m_{45} -m_{14}m_{35}+m_{15}m_{34} ,
\]
\[
 Pf(M_3) = m_{12}m_{45} -m_{14}m_{25}+m_{15}m_{24}, 
\]
\[
  Pf(M_4) = m_{12}m_{35} -m_{13}m_{25}+m_{15}m_{23},
\]
\[
 Pf(M_5) = m_{12}m_{34} -m_{13}m_{24}+m_{14}m_{23}.
\]
\end{example}

\begin{definition}
Assume that $R=k[x_1,\dots,x_n]$ is the polynomial ring in $n$ variables over a field $k$ and $I$ $\subset$ $ R$  is an ideal of $R$.
The \textit{codimension} of $I$ is given by $\text{ codim I}=  n-  \dim R/I$, where $\dim R/I$ is the Krull dimension of the ring $R/I$. 
\end{definition}

\begin{definition}
A Noetherian local ring $R$ is a \textit{Gorenstein ring} if it has finite injective dimension as  $R$-module. More generally, a Noetherian ring $R$ is called Gorenstein  if for every maximal ideal  $\mathfrak m$ of $R$ the localization $R_{\mathfrak m}$ 
is Gorenstein.

An ideal $I$ of a Gorenstein ring $R$ is called Gorenstein if the quotient ring $R/I$ is Gorenstein.
\end{definition}

\begin{theorem} \label{thm!serbusc}
Let $R=k[x_1,\dots, x_n] / I$  be the polynomial ring in $n$ variables divided by a homogeneous ideal $I$.
\begin{enumerate} 
\item (Serre)  If  $\text{codim I}=  1$ or  $2$  then
\begin{center}
  $R$ is Gorenstein $\Leftrightarrow$  $I$ is a complete intersection.
\end{center}

\,

\item (Buchsbaum-Eisenbud \cite{BE}) If  $\text{codim I}= 3$  then
\begin{center}

$ R$ is Gorenstein $\Leftrightarrow$  $I$ is generated by the $2n\times 2n$ pfaffians
  of a skewsymmetric $ (2n+1)\times (2n+1)$
matrix with entries in $R$.
\end{center}

\end{enumerate}
\end{theorem}

For a proof of Theorem \ref{thm!serbusc} see e.g.,  \cite[ Corollary~ 21.20]{E} and \cite[ Theorem~ 3.4.1]{BH}.

\begin{definition}\label{def!basket}
Let $X$ be a quasi-projective variety over  $\mathbb{C}$ and $x,y,z$ be coordinates of $\mathbb{C}^3$. Suppose that the group ${\mathbb{Z}}_{r}$ of rth roots of unity acts on  $\mathbb{C}^3$ via:
\begin{center}
$(x,y,z)\mapsto (\epsilon^{a}x, \epsilon^{b}y, \epsilon^{c}z )$,
\end{center}
where  $\epsilon$ is a primitive fixed rth root of unity and  $a, b, c$ are integers. A singularity $P\in X$ is a \textit{quotient singularity of type $ \frac{1}{r}(a,b,c)$} if $(X,P)$ is isomorphic to an analytic neighborhood of  $ (\mathbb{C}^3,0)/{\mathbb{Z}}_{r}$. A \textit{basket of singularities} is a collection of quotient singularities  of type $\{ \frac{1}{r_1}(a_1,b_1,c_1), \frac{1}{r_2}(a_2,b_2,c_2),\dots,$ 
$\frac{1}{r_s}(a_s,b_s,c_s)\}.$
\end{definition}

\begin{definition} \label{def!fano3fol}
 A \textit{Fano $3$-fold} is a normal projective variety $X$ of dimension $3$ which satisfies the following conditions:
\begin{enumerate}
\item  the anticanonical divisor $-K_{X}$ is ample

\item $X$ belongs to  the Mori category, that is it has at worst $\mathbb{Q}$-factorial terminal singularities.
\end{enumerate}
The maximal integer $f$ such that  $-K_{X}$ is divisible by some Weil divisor $A$, $ -K_{X}= fA$ is called the \textit{Fano index} of $X$. The Weil divisor $A$ for which $ -K_{X}= fA$ is called a \textit{primitive ample divisor}. The graded ring  $R(X, A)=\bigoplus \nolimits _{n \geq 0} H^0(X,\mathcal{O}_ X(nA))$ of a Fano $3$-fold $X$ with a primitive ample divisor $A$  is Gorenstein,  finitely generated and $X \cong   \text{Proj} \  R(X, A) $. The Hilbert Series $P_{X,A}(t)$ of $(X,A)$ is defined to be the Hilbert Series of the graded ring  $R(X,A)$. 
\end{definition}

In the case of Fano $3$-folds of index $f$,  Suzuki (\cite[ Lemma ~ 1.2]{SU}) proves that in Definition~\ref{def!basket} we can assume that $b=-a$, $c=f$ and $r$ is coprime to $a,b,c$. Therefore, for a Fano $3$-fold of index $f$,  a basket of singularities is a collection of singularity germs $ \frac{1}{r}(a,-a,f)$.

\section{Unprojection Review}

\label{sec!general_notations}

In this section, we recall some basic facts related to Kustin-Miller unprojection. For more details see \cite{P1, P2, PR, R1}.

\subsection{Kustin-Miller Unprojection}
 Assume that $R$ is a Gorenstein local ring and $J\subset R$ is a codimension $1$ ideal such that the quotient ring $R/J$ be Gorenstein.
Then,  $\operatorname{Hom}_R(J,R)$ is generated as an $R$-module by the inclusion map $i\colon J\rightarrow  R$ and an extra homomorphism $\phi \colon J\rightarrow  R$.

\begin{definition} \label{Def!kunmilring}
We define as  \textit{Kustin-Miller unprojection ring},  $\operatorname{ Unpr}(J,R)$, of the pair $J\subset R$ the graph of $\phi$,  that is the quotient 
\[ \operatorname{Unpr}(J,R) =\frac{R[T]}{(Tr-\phi(r): r \in J)},\]

where $T$ is a new variable, the unprojection variable.
\end{definition}

\begin{theorem}(Papadakis-Reid) 
The ring  $\operatorname{Unpr}(J,R)$  is Gorenstein.
\end{theorem}

The simplest example of Kustin-Miller  unprojection which nevertheless has important consequences in birational geometry is the 
example of a hypersurface which contains a codimension $2$ complete intersection.

\begin{example} (Reid's $Ax-By$ argument)

\noindent Assume that
\[ R=k[x,y,z,w]/(Ax-By),  \ \ \ \  A,B \in k[x,y,z,w]\]
and
\[J=(x,y) \subset R.\]
Then, $\operatorname{Hom}_R(J,R)$ is generated as $R$-module by $i$ and $\phi$. The $R$-module homomorphism $\phi\colon J\rightarrow  R$ is the unique homomorphism such that  $\phi(x)=B$ and  $\phi(y)=A$. 

\noindent Moreover,
 \[\operatorname{Unpr}(J,R) = R[T] / (Tx-B, Ty-A).\]
 \end{example}

\begin{remark}
\begin{enumerate}
\item  Geometrically, $\mathop{\mathrm{Spec}} (\operatorname{Unpr}(J,R))$ is birational to  $\mathop{\mathrm{Spec}}( R)$ and  $V(J)$ has been contracted.
\item The ring  $\operatorname{Unpr}(J,R)$ has typically more complicated structure than both $R$, $R/J$.
\item The ring  $\operatorname{Unpr}(J,R)$ is useful to construct and analyse Gorenstein rings in terms of simpler ones.
\item The $R$-module 
        \[ \operatorname{Hom}_R(J,R)= \{ f\colon J\rightarrow R \  \  |  f \  \    R-homomorphism\} \]
contains important information related to the birational geometry of $\mathop{\mathrm{Spec}}( R)$ with respect to $V(J)$.
\end{enumerate}
\end{remark}

\subsection{Parallel Kustin-Miller Unprojection} \label{subs!papnevthm}
 Sometimes, especially for applications, it is necessary to perform not only one but several  Kustin-Miller unprojections. Neves and Papadakis \cite{NP2} develop a theory which is called parallel Kustin-Miller unprojection. More presicely, they set sufficient conditions on a  positively graded Gorenstein ring $R$ and a finite set of codimension $1$ ideals which ensure the series of unprojections. Furthermore, they give a simple and explicit description of the end product ring which corresponds to the unprojection of the ideals. In this subsection, we recall their formulation. 

Let $\mathcal{L}$ be  a nonempty finite indexing set and $\mathcal{M}\subset \mathcal{L}$ be a nonempty subset.
Assume that  $R$ is a Gorenstein positively graded ring  and $\{J_{\alpha},  \alpha \in \mathcal{L}\}$ is a set of  codimension $1$ homogeneous ideals of $R$ such that, for all $\alpha \in \mathcal{L}$, the quotient ring  $ R/J_{\alpha}$  is Gorenstein.

We fix graded $R$-module homomorphisms $\phi_{\alpha}\colon J_{\alpha} \rightarrow R$ such that    $\operatorname{Hom}_R(J_{\alpha},R)$ is generated as an $R$-module by  $\{i_{\alpha},\phi_{\alpha}\}$, where $i_{\alpha}\colon J_{\alpha}\rightarrow R$ is the inclusion map.
Assume that for distinct $\alpha$, $\beta$ $\in \mathcal{L}$  there exists a homogeneous element $r_{\alpha\beta}\in R$ with $\deg  \    r_{\alpha\beta} = \deg \  \phi_{\alpha}$ such that
\begin{equation}
 (\phi_{\alpha}+r_{\alpha\beta}i_\alpha)(J_{\alpha})\subset J_{\beta}
\end{equation}
and for all distinct $\alpha , \beta \in \mathcal{L}$,
\begin{equation}
\text{codim}_R(J_{\alpha} + J_{\beta})\geq 2.
\end{equation}
Denote by $\phi_{\alpha\beta}= \phi_{\alpha}+r_{\alpha\beta}i_\alpha$. Fix distinct $\alpha$, $\beta$ $\in \mathcal{L}$. 
Then, there exists a unique homogeneous element $A_{\beta\alpha}\in R$ of degree  $\deg \  \phi_{\alpha} +  \deg \  \phi_{\beta}$   such that
\[\phi_{\beta\alpha}(\phi_{\alpha\beta}(s))= A_{\beta\alpha}s \ \ 
\text{for all} \ s\in J_{\alpha}. \]

Assume that  $\{T_u | u\in \mathcal{M}\}$ is a set of new variables with degree of $T_u$  equal to  $\deg \  \phi_{u}$ for all $u\in \mathcal{M}$.
Denote by $R_{\mathcal{M}}$   the graded ring given as quotient of polynomial ring  $ R[ T_{u}  |   u\in \mathcal{M}]$ by the ideal generated by the set
\[ \{T_{u}s-\phi_{u}(s) | u\in \mathcal{M}, s\in J_{u} \} \cup  \{(T_{v}+r_{vu})(T_{u}+r_{uv})-A_{vu} | u,v\in \mathcal{M}, u\neq v \}.\]
Consider the natural ring homomorphism  $\pi \colon R[  T_{u} |  u\in \mathcal{M}] \rightarrow R_{\mathcal{M}}$. 
Given $w\in \mathcal{L} \setminus \mathcal{M}$, we denote by $J_{\mathcal{M},w}\subset R_{\mathcal{M}}$ the ideal generated by the image under $\pi$ of the subset
\[  J_{w}\cup \{T_{u}+r_{uw} | u\in  \mathcal{M} \}.  \]

\begin{theorem} ({\bf Neves-Papadakis}) \label{thm!nevespapadakis} 
\begin{enumerate}
\item  The ring $R_\mathcal{M}$ is Gorenstein with dimension equal to the dimension of $R$. Moreover, the natural map  $R\rightarrow R_{\mathcal{M}}$ is injective.
\item Assume there exists $w \in \mathcal{L} \setminus \mathcal{M}$. The map $\phi_{w}\colon J_{w}\rightarrow R$ is extended to an $R_{\mathcal{M}}$-homomorphism 
$\Phi_{\mathcal{M},w}\colon J_{\mathcal{M},w}\rightarrow R_{\mathcal{M}}$ uniquely determined by the property 
 \[\Phi_{\mathcal{M},w}(T_{u}+r_{uw})=A_{uw}-(T_{u}+r_{uw})r_{wu},\]
for all $u\in \mathcal{M}$.
\item  Under the assumptions of $(2)$, the codimension of the ideal  $J_{\mathcal{M},w}$ is equal to $1$ and the quotient ring $R_{\mathcal{M}}/J_{\mathcal{M},w}$ is Gorenstein.
The $R_{\mathcal{M}}$-module $\operatorname{Hom}_{R_{\mathcal{M}}}(J_{\mathcal{M},w}, R_{\mathcal{M}} )$  is generated by the natural inclusion $ i_{\mathcal{M},w}\colon J_{\mathcal{M},w}\rightarrow R_{\mathcal{M}}$   and the $R_{\mathcal{M}}$-homomorphism   $\Phi_{\mathcal{M},w}$. The ring   $R_{\mathcal{M}\cup  \{w\}}$  is the Kustin-Miller unprojection ring (in the sense of the Definition ~ \ref{Def!kunmilring}) of the pair $J_{\mathcal{M},w}\subset R_{\mathcal{M}}$.                                   
\end{enumerate}

\end{theorem} 

\subsection{Tom and Jerry Unprojections} \label{Sec!Tomjer}

Assume that  $M=(m_{kl})$ is a $5\times5$ skewsymmetric matrix,  $I$ is the codimension $3$ ideal given by the pfaffians of $M$ and $J$  is a codimension $4$ complete intersection ideal.

We fix the ideal $J$. The question is to find suitable matrices $M$ such that $I$ is a subset of $J$. Tom and Jerry are two different answers to this question. 

Tom and Jerry are two families of unprojections which were originally defined and named by Reid. These families occur in many constructions of Gorenstein codimension $4$ ideals with $9\times16$ resolution (i.e., $9$ equations and $16$ first syzygies) and they can be considered as a type of deformation of the homogeneous coordinate rings of the Segre embedding  $\mathbb{P}^2 \times   \mathbb{P}^2 \subset  \mathbb{P}^8$ and $\mathbb{P}^1\times   \mathbb{P}^1\times   \mathbb{P}^1\subset  \mathbb{P}^7$ respectively.

Assume   $1\leq i\leq 5$. The matrix $M$ is called  $Tom_i$  in $J$ if the $m_{kl}$ entry of  $M$ belongs to the ideal $J$ whenever neither $k$  nor $l$ is equal to $i$. 
In other words, if we delete the ith row and ith column of $M$ the remaining entries are elements of the codimension $4$ ideal $J$. For an example of $M$ which is Tom$_1$ in $J=(z_1,z_2,z_3,z_4)$ see Subsection \ref{fund!calc}.

Assume  $1\leq i < j \leq 5$. The matrix $M$ is called  $Jerry_{ij}$ in $J$ if the $m_{kl}$ entry of  $M$  belongs to the ideal $J$ whenever either $k$ or $l$ is equal to $i$ or $j$.
In other words, all the entries of $M$ that belong to the ith row or ith column or jth row or jth column are elements of $J$ while the remaining entries of $M$ are not restricted.

\begin{remark}\label{rem!elemrow}
Assume $M$ is a Tom$_i$ matrix in $J$. Then with a sequence of elementary row and column operations we may get a matrix $M_1$ which is a Tom$_j$ matrix in $J$.

For example,  consider a matrix M which is Tom$_2$  matrix in J. Denote by $A$ the following invertible $5\times 5 $ matrix
\[
A =  \begin{pmatrix} 
0 &  1 &  0 &  0 &  0\\
1 & 0  & 0 & 0& 0\\
0& 0  & 1 & 0&  0\\
0 & 0  & 0 & 1&  0\\
0 & 0  & 0 & 0&  1
\end{pmatrix}
\]
The matrix $A M A^{t}$, where $A^{t}$ is the transpose of $A$, is a matrix which is Tom$_1$ in J.
Similarly, we can relate a matrix $M$ which is a Jerry$_{ij}$ matrix in $J$ to  a matrix $M_1$ which is a Jerry$_{i'j'}$ matrix in $J$.

\end{remark}

\subsection{Papadakis  Fundamental Calculation for Tom} \label{fund!calc}
Papadakis  \cite{P1} gives an explicit presentation of the unprojection ring for the Tom and Jerry families. In what follows,  we give a quick review of  the main steps of Papadakis fundamental calculation for Tom. The following matrix $N$ is a matrix which is Tom$_1$ in the codimension $4$ complete intersection ideal $J=(z_1,z_2,z_3,z_4)$.

We work over the  polynomial ring   $ R=k[x_k, z_k, m_{ij}^k]$  where $1\leq k \leq 4$, $2\leq i < j \leq 5$.
Denote   by
\[
N =  \begin{pmatrix} 
0 & x_{1} & x_{2} & x_{3}& x_{4}\\
-x_{1} & 0  & m_{23} & m_{24}& m_{25}\\
-x_{2} & -m_{23}  & 0 & m_{34}&  m_{35}\\
-x_{3} & -m_{24}  & -m_{34} & 0&  m_{45}\\
-x_{4} & -m_{25}  & -m_{35} & -m_{45}&  0
\end{pmatrix},
\]
where  
\[
 m_{ij}= \sum_{k=1}^{4} m_{ij}^k z_k.
\]

Let $ I$ be the ideal given by the pfaffians $P_0 , P_1 , P_2 , P_3 , P_4 $  of $N$. We can easily prove that $I \subset J.$

It holds that  $ P_1 , P_2 , P_3 , P_4$ are linear in  $z_1, z_2, z_3, z_4$. Then, there exists a unique  $4\times 4$ matrix $Q$ such that
\[
 \begin{pmatrix} 
P_{1} \\ P_{2} \\ P_{3} \\ P_{4}
\end{pmatrix} =   Q  \begin{pmatrix} 
z_{1} \\ z_{2} \\  z_{3} \\ z_{4}
\end{pmatrix}.
\]

We denote by  $Q_{i}$ the submatrix of $Q$ which occurs by deleting the ith row of  $Q$. For  $i = 1,\dots, 4$, let  $ H_i$  be the  $1\times 4$ matrix whose ith entry is equal to $ (-1)^{i+1}$ times the determinant of the submatrix of $Q_{i}$ obtained by deleting the ith column. For all $i, j$,  it holds that

\begin{center}
$x_{i}  H_{j} = x_{j}  H_{i}.$
\end{center}
Using the last equality, we can define four polynomials $g_1, g_2, g_3, g_4$ as follows. We fix \, $1\leq j\leq 4$ and we set

\begin{center}
$(g_1, g_2, g_3, g_4) =   H_{j}/x_{j}$.
\end{center}
We note that this definition is independent of the choice of $j$.

Denote by $\phi$ the map which is defined by
\begin{center}
$\phi \colon J/I \rightarrow R/I$
\end{center}

\begin{center}
 $ z_i \mapsto  g_i$.
\end{center}
It holds that $\operatorname{Hom}_{R/I}(J/I,R/I)$ is generated as $R/I$-module by the inclusion map $i$ and $\phi$. Moreover, the ideal 
\begin{center}
$(P_0 ,  P_1,   P_2,  P_3,  P_4,    Tz_1-g_1, Tz_2-g_2,   Tz_3-g_3,   Tz_4-g_4)$
\end{center}
of the polynomial ring $R[T]$ is Gorenstein of codimension $4$.

\,

\,

\section{Tom and Jerry Triples}

In this section we describe some ways to put a codimension $3$ ideal $I$ defined by the pfaffians of a $5\times 5$ skewsymmetric matrix $M$ inside three codimension $4$ complete intersection ideals $J_1, J_2, J_3$. One of the ways is studied in detail in Section \ref{Sec!mainresult} and leads to an application to Fano $3$-folds in Section \ref{Sec!app}. We have checked with the computer algebra program Macaulay2~\cite{GS} that also the remaining cases lead to construction of codimension $6$ Gorenstein ideals after a suitable choice of $J_1$, $J_2$, $J_3$.

Assume that 
\[
M=  \begin{pmatrix} 
0 & m_{12} & m_{13} & m_{14}& m_{15}\\
-m_{12} & 0  & m_{23} & m_{24}& m_{25}\\
-m_{13} & -m_{23}  & 0 & m_{34}&  m_{35}\\
-m_{14} & -m_{24}  & -m_{34} & 0&  m_{45}\\
-m_{15} & -m_{25}  & -m_{35} & -m_{45}&  0
\end{pmatrix}
\]
is a $5\times 5$ skewsymmetric matrix and $J_{1}$, $J_{2}$, $J_{3}$ are three complete intersection ideals of  codimension $4$. In each of the following cases we set conditions in the entries of $M$ such that the ideal $I$ of pfaffians is contained in each of the ideals $J_1, J_2, J_3$. We denote by $S_{5}$ the symmetric group of permutations of the set $\{1\dots 5\}$.

\subsection{ Tom $\&$  Tom $\&$ Tom Case}  \label{subs!the_main_players}
We define the following equivalence relation

\begin{center}
$ Tom_{i} +  Tom_{j} + Tom_{k}   \sim   Tom_{i'} +  Tom_{j'} + Tom_{k'} $, 
\end{center}
for $1\leq  i<  j<  k \leq 5$ and $1\leq  i'<  j'<  k' \leq 5$,  if there exists $\sigma \in S_{5}$ such that

\begin{center}
$\sigma(i)= i'$, \,  $\sigma(j)= j'$, \,  $\sigma(k)= k'.$

\end{center}

It is not difficult to see that there is only one equivalence class with representative  the element    $Tom_{1} + Tom_{2}+ Tom_{3}$.

\begin{definition}\label{gen!M}
 We say that $M$ is a  \textit{$Tom_{1} + Tom_{2}+ Tom_{3}$}  matrix if the entries of $M$ satisfy the following conditions:
\begin{center}
$m_{12}\in J_{3},$   \,  $m_{13}\in J_{2},$  \,  $m_{14},  m_{15}\in J_{2}\cap J_{3},$  \,  $m_{23}\in J_{1},$
\end{center}
\begin{center}
$m_{24},  m_{25}\in J_{1}\cap J_{3},$  \, $m_{34},  m_{35}\in J_{1}\cap J_{2},$  \,  $ m_{45}\in J_{1}\cap J_{2}\cap J_{3}$.
\end{center}
Then,  $M$ is Tom$_1$ in $J_1$, Tom$_2$ in $J_2$ and Tom$_3$ in $J_3$.
\end{definition} 

\subsection{  Jerry $\&$  Jerry  $\&$ Jerry Case}
Working as before, we define the following equivalence relation 

\begin{center}
 $ Jerry_{ij} +  Jerry_{kl} + Jerry_{mn}   \sim   Jerry_{i'j'} +  Jerry_{k'l'} + Jerry_{m'n'}$, 
\end{center}
for $1\leq  i<  j\leq 5$, $1\leq  k<  l\leq 5$, $1\leq  m<  n\leq 5$ and $(i,j), (k,l), (m,n)$  pairwise different and similarly for the indices $i',j',k',l',m',n'$,  if there exists $\sigma \in S_{5}$ such that

\begin{center}
$\sigma{(i)}= i'$, \,  $\sigma{(j)}= j'$, \,  $\sigma{(k)}= k'$, \,  $\sigma{(l)}= l'$, \,  $\sigma{(m)}= m'$, \,  $\sigma{(n)}= n'.$

\end{center}

In this case the following equivalence classes occur:

\begin{enumerate}
\item   $Jerry_{ij} + Jerry_{il} + Jerry_{in}$  with representative the element   $Jerry_{12} + Jerry_{13} + Jerry_{14}$
\item   $Jerry_{ij} + Jerry_{il} + Jerry_{jl}$  with representative the element   $Jerry_{12} + Jerry_{13} + Jerry_{23}$
\item   $Jerry_{ij} + Jerry_{il} + Jerry_{jn}$  with representative the element   $Jerry_{12} + Jerry_{14} + Jerry_{23}$
\item   $Jerry_{ij} + Jerry_{il} + Jerry_{mn}$  with representative the element  $Jerry_{14} + Jerry_{15} + Jerry_{23}$.
\end{enumerate}

\begin{definition} \label{def!3.3}
 We say that $M$ is a  

\begin{enumerate}
\item
\textit{$Jerry_{12} + Jerry_{13}+ Jerry_{14}$}  matrix if the entries of $M$ satisfy the following conditions:
\begin{center}
$m_{12}, m_{13}, m_{14},m_{15} \in  J_{1}\cap J_{2}\cap J_{3},$  \, $m_{23}\in J_{1}\cap  J_{2},$  \,  $m_{24}\in J_{1}\cap J_{3},$
\end{center}
\begin{center}
$m_{25}\in J_{1},$  \,   $m_{34}\in J_{2}\cap J_{3},$  \,  $m_{35}\in  J_{2},$ \,  $ m_{45}\in J_{3}$.
\end{center}
Then, M is Jerry$_{12}$ in $J_1$, Jerry$_{13}$ in $J_2$ and  Jerry$_{14}$ in $J_3$. A similar comment is true in the definitions that follow and we will not write it explicitly.
\item 
\textit{$Jerry_{12} + Jerry_{13}+ Jerry_{23}$}  matrix if the entries of $M$ satisfy the following conditions:
\begin{center}
$m_{12}, m_{13}\in  J_{1}\cap J_{2}\cap J_{3},$ \,  $m_{14},m_{15} \in  J_{1}\cap J_{2},$\, $m_{23}\in J_{1}\cap  J_{2}\cap J_{3},$
\end{center}
\begin{center}
$m_{24},m_{25}\in J_{1}\cap J_{3},$ \,  $m_{34},m_{35}\in J_{2}\cap J_{3},$ \, $ m_{45}$:  unconstrained.
\end{center}
\item 
\textit{$Jerry_{12} + Jerry_{14}+ Jerry_{23}$}  matrix if the entries of $M$ satisfy the following conditions:
\begin{center}
$m_{12}, m_{13}\in  J_{1}\cap J_{2}\cap J_{3}$, \,  $m_{14},m_{15} \in  J_{1}\cap J_{2}$, \,  $m_{23}\in J_{1}\cap J_{3}$, \,
\end{center}
\[
\  \  \  \, \, \, \, \  \,   m_{24}\in J_{1}\cap J_{2}\cap J_{3}, \, m_{25}\in J_{1}\cap J_{3}, \, m_{34}\in J_{2}\cap J_{3}, \,   
 m_{35}\in  J_{3}, \,  m_{45}\in  J_{2}.
\]
\item 
\textit{$Jerry_{14} + Jerry_{15}+ Jerry_{23}$}  matrix if the entries of $M$ satisfy the following conditions:
\begin{center}
$m_{12}, m_{13}\in  J_{1}\cap J_{2}\cap J_{3},$ \, $m_{14},m_{15} \in  J_{1}\cap J_{2},$ \, $m_{23}\in J_{3},$
\end{center}
\begin{center}
$m_{24}\in J_{1}\cap J_{3},$ \,  $m_{25}\in J_{2}\cap J_{3},$ \, $m_{34}\in J_{1}\cap J_{3},$ \, $m_{35}\in J_{2}\cap J_{3},$ \,  $ m_{45}\in  J_{1}\cap J_{2}$.
\end{center}
\end{enumerate}
\end{definition}

\,
\subsection{  Tom $\&$  Tom  $\&$ Jerry Case}
In this case  we define the following equivalence relation

\begin{center}
 $ Tom_{i} +  Tom_{j} + Jerry_{kl}   \sim   Tom_{i'} +  Tom_{j'} + Jerry_{k'l'}$, 
\end{center}
for $1\leq  i<  j\leq 5$, $1\leq  k<  l\leq 5$  and similarly for the indices $i',j',k',l'$,  if there exists $\sigma \in S_{5}$ such that

\begin{center}
$\sigma{(i)}= i'$, \,  $\sigma{(j)}= j'$, \,  $\sigma{(k)}= k'$, \,  $\sigma{(l)}= l'$.  

\end{center}

So, the following equivalence classes arise:

\begin{enumerate}
\item   $Tom_{i} + Tom_{j} + Jerry_{ij}$  with representative the element   $Tom_{1} + Tom_{2} + Jerry_{12}$
\item   $Tom_{i} + Tom_{j} + Jerry_{il}$  with representative the element   $Tom_{1} + Tom_{2} + Jerry_{13}$
\item   $Tom_{i} + Tom_{j} + Jerry_{kl}$  with representative the element   $Tom_{1} + Tom_{2} + Jerry_{34}$.
\end{enumerate}

\begin{definition}\label{def!3.5}
We say that $M$ is a  
\begin{enumerate}
\item
\textit{$Tom_{1} + Tom_{2}+ Jerry_{12}$}  matrix if the entries of $M$ satisfy the following conditions:
\begin{center}
$m_{12}\in  J_{3},$ \, $m_{13}, m_{14}, m_{15}\in   J_{2}\cap  J_{3},$ \, $m_{23}, m_{24}, m_{25}\in J_{1}\cap  J_{3},$ $m_{34}, m_{35}, m_{45}\in J_{1}\cap J_{2}$.
\end{center}
\item 
\textit{$Tom_{1} + Tom_{2}+ Jerry_{13}$}  matrix if the entries of $M$ satisfy the following conditions:
\begin{center}
$m_{12}\in  J_{3},$ \,  $m_{13}, m_{14},m_{15} \in  J_{2}\cap J_{3},$ \,  $m_{23}\in J_{1}\cap  J_{3},$
\end{center}
\begin{center}
$m_{24},m_{25}\in J_{1},$ \, $m_{34},m_{35}\in J_{1}\cap J_{2}\cap J_{3},$ \, $ m_{45}\in J_{1}\cap J_{2}$.
\end{center}
\item 
\textit{$Tom_{1} + Tom_{2}+ Jerry_{34}$}  matrix if the entries of $M$ satisfy the following conditions:
\begin{center}
$m_{12}$ :  unconstrained, \,  $m_{13},m_{14} \in  J_{2}\cap J_{3},$ \,  $m_{15}\in  J_{2},$
\end{center}
\begin{center}
$m_{23}, m_{24}\in J_{1}\cap J_{3},$ \, $m_{25}\in J_{1},$ \,  $m_{34},m_{35}, m_{45} \in  J_{1}\cap J_{2}\cap J_{3}$.
\end{center}
\end{enumerate}
\end{definition} 

\subsection{   Tom $\&$  Jerry  $\&$ Jerry Case}
In this last case  we define the following equivalence relation 
\begin{center}
 $ Tom_{i} +  Jerry_{jk} + Jerry_{lm}   \sim   Tom_{i'} +  Jerry_{j'k'} + Jerry_{l'm'}$ , 
\end{center}
for  $1\leq  i\leq 5$,   $1\leq  j<  k\leq 5$, $1\leq  l<  m\leq 5$ and $(j,k)\neq (l,m)$ and similarly for the indices $i',j',k',l',m'$, if there exists $\sigma \in S_{5}$ such that
\begin{center}
$\sigma{(i)}= i'$, \,  $\sigma{(j)}= j'$, \,  $\sigma{(k)}= k'$, \,  $\sigma{(l)}= l'$, \,  $\sigma{(m)}= m'$. 

\end{center}

So, we have the following equivalence classes:

\begin{enumerate}
\item $Tom_{i} + Jerry_{ik} + Jerry_{im}$ with representative the element $Tom_{1} + Jerry_{12} + Jerry_{13}$
\item $Tom_{i} + Jerry_{ik} + Jerry_{km}$ with representative the element $Tom_{1} + Jerry_{12} + Jerry_{23}$
\item $Tom_{i} + Jerry_{ik} + Jerry_{lm}$ with representative the element  $Tom_{1} + Jerry_{12} + Jerry_{34}$
\item $Tom_{i} + Jerry_{jk} + Jerry_{jm}$ with representative the element $Tom_{1} + Jerry_{23} + Jerry_{24}$
\item $Tom_{i} + Jerry_{jk} + Jerry_{lm}$  with representative the element $Tom_{1} + Jerry_{23} + Jerry_{45}$.
\end{enumerate}

\begin{definition} \label{def!3.7}
We say that $M$ is a  
\
\begin{enumerate}
\item
\textit{$Tom_{1} + Jerry_{12} + Jerry_{13}$}  matrix if the entries of $M$ satisfy the following conditions:
\begin{center}
$m_{12}, m_{13}, m_{14}, m_{15}\in   J_{2}\cap  J_{3},$ \,  $m_{23}\in   J_{1}\cap J_{2}\cap  J_{3},$ \, $ m_{24}, m_{25}\in J_{1}\cap  J_{2},$
\end{center}
\begin{center}
$m_{34}, m_{35}\in J_{1}\cap J_{3},$ \, $ m_{45}\in J_{1}$.
\end{center}
\,
\item 
\textit{$Tom_{1} + Jerry_{12} + Jerry_{23}$}  matrix if the entries of $M$ satisfy the following conditions:
\begin{center}
$m_{12}, m_{13}\in  J_{2}\cap J_{3},$ \, $ m_{14},m_{15} \in  J_{2},$ \, $m_{23},m_{24},m_{25} \in J_{1}\cap  J_{2}\cap  J_{3},$
\end{center}
\begin{center}
$m_{34},m_{35}\in J_{1}\cap J_{3},$ \,  $ m_{45}\in J_{1}$.
\end{center}
\,
\item 
\textit{$Tom_{1} + Jerry_{12} + Jerry_{34}$}  matrix if the entries of $M$ satisfy the following conditions:
\begin{center}
$m_{12}\in J_{2},$ \, $m_{13},m_{14} \in  J_{2}\cap J_{3},$ \, $m_{15}\in  J_{2},$ \, $m_{23}, m_{24}\in J_{1}\cap J_{2}\cap J_{3},$ \\
$m_{25}\in J_{1}\cap J_{2},$ \, $m_{34},m_{35}, m_{45} \in  J_{1}\cap J_{3}.$
\end{center}
\,
\item 
\textit{$Tom_{1} + Jerry_{23} + Jerry_{24}$}  matrix if the entries of $M$ satisfy the following conditions:
\begin{center}
$m_{12}\in J_{2}\cap J_{3},$  \,  $m_{13}\in  J_{2},$ \,  $m_{14}\in  J_{3},$ \,  $m_{15}$ :  unconstrained, \, $m_{23}, m_{24}, m_{25}, m_{34}\in J_{1}\cap J_{2}\cap J_{3},$  \,  $m_{35}\in J_{1}\cap J_{2},$ \, $ m_{45} \in  J_{1}\cap J_{3}$.
\end{center}

\newpage

\item 
\textit{$Tom_{1} + Jerry_{23} + Jerry_{45}$}  matrix if the entries of $M$ satisfy the following conditions:
\begin{center}
$m_{12}, m_{13}\in J_{2},$  \,  $m_{14}, m_{15}\in  J_{3},$ \, $m_{23}\in J_{1}\cap J_{2},$ 
\end{center}
\begin{center}
$m_{24}, m_{25}, m_{34}, m_{35}\in J_{1}\cap J_{2}\cap J_{3}$, \, $ m_{45} \in  J_{1}\cap J_{3}$.
\end{center}
\end{enumerate}
\end{definition}

\section{Tom $\&$ Tom $\&$ Tom Format} \label{Sec!mainresult}
In the present section, we specify three codimension $4$ complete intersection ideals $J_{1}, J_{2}, J_{3}$  and a codimension 3 ideal $I$ generated by the pfaffians of a specific Tom$_{1}$+Tom$_{2}$+Tom$_{3}$ matrix. We prove that this data satisfies the conditions for parallel Kustin-Miller unprojection established by Neves and Papadakis and recalled in Theorem~\ref{thm!nevespapadakis}. Moreover, using Theorem~\ref{thm!nevespapadakis} we give a description of the final ring as a quotient of a polynomial ring by a codimension $6$ ideal. This format will be used in Section~\ref{Sec!app} to prove the existence of two families of codimension $6$ Fano $3$-folds.

We work over the polynomial ring  $R= k[z_i, c_j]$,  where $1\leq  i \leq 7$  and   $1\leq  j\leq 25$.
Denote by $Tom(1,2,3)$ the following  $5\times 5$ skewsymmetric matrix 
\[
\small \begin{pmatrix} 
0 &  c_1 z_1+c_2 z_2 + c_3 z_3+c_4 z_6 & c_5 z_1+c_6 z_2 +c_7 z_4 + c_8 z_5 & c_9 z_1 + c_{10} z_2 &c_{11} z_1 +c_{12} z_2 \\
    & 0  &  c_{13}  z_2+c_{14} z_3 + c_{15} z_5+ c_{16} z_7 & c_{17} z_2 +c_{18} z_3  &  c_{19} z_2 +c_{20} z_3 \\
 &  & 0 & c_{21}  z_2+c_{22} z_5  & c_{23} z_2 +c_{24} z_5 \\
 &  -Sym  &  & 0& c_{25} z_2 \\
 & &  &  &  0
\end{pmatrix}
\]
which is Tom$_{1}$+Tom$_{2}$+Tom$_{3}$  matrix in the ideals
\begin{center}
$J_1= (z_2, z_3, z_5, z_7)$,  \,   $J_2= (z_1, z_2, z_4, z_5)$,  \,  $J_3= (z_1,z_2,z_3,z_6)$.
\end{center}
Let $I$ be the ideal generated by the pfaffians of $Tom(1,2,3)$.

\begin{proposition} \label{thm!assum}
 
(i) For all $t$ with $1\leq t\leq 3$,  the ideal  $J_t / I$ is a codimension $1$ homogeneous ideal of the quotient ring $R/I$ such that the ring $R/I/ J_t/I$ is Gorenstein.

(ii) For all $t, s$ with $1\leq t < s\leq 3$,  it holds that  $\mathop{codim}_{R/I}(J_t/I+J_s/I)= 3$.

\end{proposition}

\begin{proof}
We first prove $(i)$.
According to the Third Isomorphism Theorem of rings 
\begin{equation} \label{iso!pols1} 
 R/I/J_1/I  \cong k[z_1,z_4,z_6,c_1,\dots, c_{25}], \,  \, \,     R/I/J_2/I \cong    k[z_3,z_6,z_7,c_1,\dots, c_{25}], 
\end{equation}

\begin{center}
 \label{iso!pols3} $R/I/J_3/I  \cong    k[z_4,z_5,z_7,c_1,\dots, c_{25}]$.
\end{center}
So, we conclude that for all t with $1\leq t\leq 3$,

\begin{center}
$ \dim \     R/I/J_t/I= 28.$
\end{center}
We claim that

\begin{center}

 $\dim \  R/I = 29 .$
\end{center}
Assume that  $\dim \  R/I = x. $
We denote by $\tilde{I}= (c_1,c_2,c_3,c_5,c_6,c_7,c_9,c_{12},c_{13},c_{15},c_{16},c_{18},c_{19},\\
c_{21},c_{23})$,
the ideal generated by some variables of  $R$. We set  $J^{new}= I+\tilde{I}$. The ideal $J^{new}$ is a homogeneous ideal of $R$. Hence, from Krull's principal ideal theorem it follows that 

\begin{center}
$\dim \ R/J^{new} \geq  x- 15.$
\end{center}
We call $\hat{I}$ the ideal obtained from the ideal $I$ by setting the variables  $c_1, c_2, c_3, c_5, c_6, c_7, c_9, c_{12}, \\
c_{13}, c_{15}, c_{16}, c_{18}, c_{19}, c_{21}, c_{23}$
be equal to zero. Using  the Third Isomorphism Theorem of rings  as before we have that
\[
 R/J^{new} \cong  k[z_1,\dots, z_7,c_4,c_8,c_{10},c_{11},c_{14},c_{17},c_{20},c_{22},c_{24},c_{25}]/\hat{I}.
\]
For the computation of the Krull  dimension of  
\begin{center}
$k[z_1,\dots, z_7,c_4,c_8,c_{10},c_{11},c_{14},c_{17},c_{20},c_{22},c_{24},c_{25}]/\hat{I}$
\end{center}
 we used the computer algebra program Macaulay2 \cite{GS}.

It occurs that 
\begin{center}
$\dim \  k[z_1,\dots, z_7,c_4,c_8,c_{10},c_{11},c_{14},c_{17},c_{20},c_{22},c_{24},c_{25}]/\hat{I} = 14$
\end{center}
and therefore 

\begin{center}
$ \dim \ R/J^{new} = 14.$
\end{center}
As a consequence, $x\leq 29$.   

It is well-known that, see for example \cite[Theorem~ 3.4.1(a)]{BH} the ideal generated by the pfaffians of a skewsymmetric matrix has codimension $\leq 3$.  Hence, \   $\text{codim} \  I \leq 3$.   Equivalently, by the definition of codimension

\begin{center}

$x\geq 29,$

\end{center}
which completes the proof of the claim. As a consequence,

\begin{center}
$\text{codim} \  I =  \dim \ R - \dim \ R/I = 3. $
\end{center}

Hence, by the second part of the Theorem \ref{thm!serbusc},  $R/I$ is a Gorenstein ring.
Using again the definition of codimension for all t with $1\leq t\leq 3$, we get

\begin{center}
$\text{codim} \  J_t/I=1$.
\end{center}
Due to the isomorphisms   (\ref{iso!pols1})  for all $t$ with $1\leq t\leq 3$,  the ring  $R/I/J_t/I$ is Gorenstein.

 We now prove $(ii)$. Third Isomorphism Theorem of rings implies that

\[R/I/(J_1/I+J_2/I)  \cong k[z_6,c_1,\dots, c_{25}], \, \, \, \,  R/I/(J_1/I+J_3/I) \cong  k[z_4,c_1,\dots, c_{25}],  \] 
\[R/I/(J_2/I+J_3/I) \cong  k[z_7,c_1,\dots, c_{25}].\]

From the later isomorphisms it holds that for $t, s$ with $1\leq t < s\leq 3$,
\[\dim \  R/I/(J_t/I+J_s/I)= 26.\]

We remind that $\dim \  R/I= 29$. Taking into account the definition of codimension we conclude that for all $t, s$ with $1\leq t < s\leq 3$,
\[\text{codim} \ (J_t/I+J_s/I)= 3.\]
\end{proof}

For all $t$, with $1\leq t\leq 3$, we denote by $i_t\colon  J_t/I \rightarrow R/I$ the  inclusion map. Our aim is to  define  $  \phi_t\colon  J_t/ I \rightarrow R/I$ for all $t$, with $1\leq t\leq 3$, and  prove that these maps satisfy the assumptions of the Theorem \ref{thm!nevespapadakis}. As a first step for the definition of $\phi_t$,  we relate Tom$_t$ matrix in $J_t$ (for the definition see Subsection \ref{Sec!Tomjer}) to the matrix $N$ which was defined in Subsection \ref{fund!calc}.

Assume $D$ is a Tom$_1$ matrix in $J_1$. It is clear that $D$ is a specialization of the matrix $N$. For an example see Equation~(\ref{Mat!examtom}) below.  

Assume $D$ is a Tom$_2$ matrix in $J_2$. Let $A$ be the  invertible $5\times 5 $ matrix in  Example ~ \ref{rem!elemrow}. The matrix $A  D A^{t}$, where $A^{t}$ is the transpose of $A$, is a specialization of the matrix N.  

Assume $D$ is a Tom$_3$ matrix in $J_3$. Denote by $B$ the following invertible $5\times 5 $ matrix
\[
B=  \begin{pmatrix} 
0 &  0 &  1 &  0 &  0\\
1 & 0  & 0 & 0& 0\\
0& 1  & 0 & 0&  0\\
0 & 0  & 0 & 1&  0\\
0 & 0  & 0 & 0&  1
\end{pmatrix}
\]
The matrix $B  D  B^{t}$ is a specialization of the matrix N.  

We remark that the ideal generated by the pfaffians of  the  matrix $A  D A^{t}$ is equal to the ideal generated by the pfaffians of  the  matrix $ D$. The same is true for the ideal generated by the pfaffians of  the  matrix $B  D B^{t}$.

\newpage

As a second step we apply the above considerations to obtain, for $t=1,2,3$, the matrix $Tom(1,2,3)$ in $J_t$ as a specialization of $N$.
Denote by $D_1$  the following matrix 
\begin{equation} \label{Mat!examtom}
\end{equation}
\[
\small \begin{pmatrix} 
0 &  x_1 & x_2 & x_3 & x_4 \\

    & 0  & c_1  z_2+ c_2 z_3+ c_3 z_5+ c_4 z_7  & c_5 z_2+ c_6 z_3+ c_7 z_5+ c_8 z_7  &  c_9 z_2+ c_{10} z_3 + c_{11} z_5+ c_{12} z_7\\

 &  & 0 &  c_{13} z_2+ c_{14} z_3+ c_{15} z_5+ c_{16} z_7 &  c_{17} z_2+ c_{18} z_3+ c_{19} z_5+ c_{20} z_7 \\

 &  -Sym  &  & 0&   c_{21} z_2+ c_{22} z_3+ c_{23} z_5+ c_{24} z_7 \\

 & &  &  &  0
\end{pmatrix}
\]

\,

The matrix $D_1$ is a Tom$_1$ matrix in $J_1$. $D_1$ is obtained from $N$ by the following substitutions
\begin{equation}
\label{eqn!ntod1} z_1= z_2, \,  z_2=z_3, \,  z_3=z_5, \, z_4=z_7
\end{equation}
and the obvious substitutions of $m_{ij}^k$ in terms of $c_l$.
We set 

\begin{equation}
\label{eqn!relatingtom1prim} c_7=c_8=c_{11}=c_{12}=c_{14}=c_{16}=c_{18}=c_{20}=c_{22}=c_{23}=c_{24}=0
\end{equation}
in $D_1$.  We call $D_2$  the matrix which occurs. It is given explicitly by,
\,
\,
\[
 \small \begin{pmatrix} 
0 &  x_1 & x_2 & x_3 & x_4 \\

    & 0  & c_1  z_2+ c_2 z_3+ c_3 z_5+ c_4 z_7  & c_5 z_2+ c_6 z_3  &  c_9 z_2+ c_{10} z_3 \\

 &  & 0 &  c_{13} z_2+ c_{15} z_5 &  c_{17} z_2+ c_{19} z_5 \\

 &  -Sym  &  & 0&   c_{21} z_2 \\

 & &  &  &  0
\end{pmatrix}.
\]

\,
\,
Finally, setting    

\begin{align} \label{eqn!relatingtom1totom123prim} 
 \  \    \   \   \  \   \   \  \   \   \   x_1 = c_1 z_1+ c_2 z_2+ c_3 z_3+ c_4 z_6, \, x_2 = c_5 z_1+ c_6 z_2+ c_7 z_4+ c_8 z_5, \,  x_3 = c_9 z_1+ c_{10} z_2, 
\end{align}
\begin{center}
$x_4 = c_{11} z_1+ c_{12} z_2$, \,  $c_1 = c_{13}$ , \,   $c_2 = c_{14}$, \,   $c_3 = c_{15}$ , \,  $ c_4 = c_{16}$, 
\end{center}
\begin{center}
$c_5 = c_{17} , \,  c_6 = c_{18} $, \, $ c_9 =c_{19} , \,  c_{10} = c_{20}$, \, $c_{13} = c_{21}$ , 
\end{center}
\begin{center}
$c_{15} = c_{22}$, \,  $c_{17} =c_{23}, \,    c_{19} = c_{24}$, \, $ c_{21} = c_{25}$
\end{center}
in  $D_2$ we obtain the $Tom(1,2,3)$ matrix.  A similar analysis applies to consider $Tom(1,2,3)$ matrix as a Tom$_2$ in $J_2$ and Tom$_3$ in $J_3$.

At this point we  use Papadakis Fundamental Calculation for $N$  (see Subsection ~ \ref{fund!calc}) in order to define the maps $\phi_1$, $\phi_2$ and $\phi_3$. 

Assume that $1\leq t\leq 4$. We consider the polynomial $g_t$ which was defined in Subsection ~ \ref{fund!calc}. We denote by  $g_t'$ the polynomial obtained from $g_t$ after the substitutions which are noted in Equation (\ref{eqn!ntod1}) and the obvious substitutions of $m_{ij}^k$ in terms of $c_t$. We denote by $\tilde{g_t}$ the polynomial obtained by $g_t'$ after the substitutions which are described in Equation (\ref{eqn!relatingtom1prim}). Finally, we denote by  $h_t$  the polynomial which occurs from the  polynomial  $\tilde{g_t}$ after the substitutions which are noted in Equation (\ref{eqn!relatingtom1totom123prim}).  

\begin{proposition}
There exists a unique graded homomorphism of $R/I$-modules $\phi_1\colon J_1/ I \rightarrow R/I$ such that 
\begin{center}
$\phi_1(z_2 + I)= h_1+ I, \, \, \, \,   \phi_1(z_3 + I)= h_2+ I$,  
\end{center}
\begin{center}
 $\phi_1(z_5 + I)= h_3+ I, \, \, \, \,   \phi_1(z_7 + I)= h_4+ I. $
\end{center}
\end{proposition}

\begin{proof}
It follows from \cite[Theorem ~ 5.6]{P2}.
\end{proof}

For the definitions of  $\phi_2$ and $\phi_3$ we work similarly. We omit the details. For all t with $1\leq t\leq 3$, the degree of $\phi_t$ is equal to $6$.

\begin{proposition}\label{prop!relatphi}
For all $t$ with $1\leq t\leq 3$, the $R/I$-module $ \operatorname{Hom}_{R/I}(J_t/I,R/I)$ is generated by the two elements $i_t$ and $\phi_t$.
\end{proposition}

\begin{proof}
It follows from \cite[Theorem ~ 5.6]{P2}.  
\end{proof}

For all $t, s$ with $1\leq t,s\leq 3$ and $t\neq s$, we define $r_{ts}=0$.

\begin{proposition} \label{prop!existrst}
For all $t, s$ with $1\leq t,s\leq 3$ and $t\neq s$, it holds that
\[
 (\phi_t+r_{ts}i_t)(J_t/I) \subset  J_s/I.\]

\end{proposition}

\begin{proof}
It is a direct computation using the definitions of the maps $\phi_t$.
\end{proof}

\begin{proposition}
For all $t, s$ with $1\leq t,s\leq 3$ and $t\neq s$, there exists a homogeneous element $A_{st}$ such that
\[(\phi_s + r_{st} i_s)((\phi_t + r_{ts} i_t)(p)) = A_{st}p \]
for all  $p\in J_t/I$.

\end{proposition}

\begin{proof}
It follows by Theorem \ref{thm!nevespapadakis}.
\end{proof}

\begin{remark}
We explicitly computed the elements $A_{st}$ using the computer algebra program Macaulay2~\cite{GS}.
\end{remark}

\begin{definition} \label{def!mrin}
Let  $T,S,W$ be three new variables of degree $6$. Following Subsection \ref{subs!papnevthm}, we define as  \textit{$I_{un}$} the  ideal 
\begin{center}
$  (I)+  (Tz_2-\phi_1(z_2),Tz_3-\phi_1(z_3),Tz_5-\phi_1(z_5), Tz_7-\phi_1(z_7),
\  \  \  Sz_1-\phi_2(z_1), Sz_2-\phi_2(z_2), Sz_4-\phi_2(z_4), Sz_5-\phi_2(z_5),
Wz_1-\phi_3(z_1), 
Wz_2-\phi_3(z_2), Wz_3-\phi_3(z_3), Wz_6-\phi_3(z_6), 
(T+r_{12})(S+r_{21})-A_{12}, (T+r_{13})(W+r_{31})-A_{13}, (S+r_{23})(W+r_{32})-A_{23} )$,
\end{center}
of the polynomial ring  $R[T,S,W]$.
We set  $ \textit{$R_{un}$}= R[T,S,W]/ I_{un}$.
\end{definition}

\begin{theorem}\label{main!thm}
 The ideal $I_{un}$ is a codimension $6$ ideal with $20$ generators and the final ring  $R_{un}$ is Gorenstein.
\end{theorem}

\begin{proof}
We remind that in Kustin-Miller unprojection codimension is increasing by $1$. Hence, the homogeneous ideal $I_{un}$, as a result of a series of three unprojections of Kustin-Miller type starting by the codimension $3$ ideal $I$, is a codimension $6$ ideal. We observe that the ideal $I_{un}$ has $20$ generators. By Propositions  ~\ref{thm!assum}, ~ \ref{prop!relatphi} and  \ref{prop!existrst},  the assumptions of Theorem~\ref{thm!nevespapadakis} are satisfied. Hence, the ring  $R_{un}$  is Gorenstein.
\end{proof}

\section{Applications}\label{Sec!app}
In this section,  we  prove  using Theorem \ref{main!thm} the existence of $2$ families of Fano $3$-folds of codimension $6$ in weighted projective space.  

The first construction is summarised in the following theorem. It corresponds to the entry~$14885$ of  Brown's Graded Ring Database \cite{BR}. More details for the construction are given in Subsection~\ref{constr!1}.  

\begin{theorem}\label{constr!14885}
There exists a family of quasi-smooth, projectively normal and projectively Gorenstein Fano $3$-folds $X\subset \mathbb{P}(1^3, 2^7)$, nonsingular away from eight quotient singularities $\frac{1}{2}(1,1,1)$, with Hilbert series

\begin{center}
$P_{X}(t)=\frac{1-20t^4+64t^6-90t^8+64t^{10}-20t^{12}+t^{16}}{(1-t)^3(1-t^2)^7}.$
\end{center}
\end{theorem}

The second construction is summarised in the following theorem. It corresponds to the entry~$12979$ of  Brown's Graded Ring Database. More details for the construction are given in Subsection~\ref{construct!id12979}.  

\begin{theorem}\label{constr!12979}
There exists a family of quasi-smooth, projectively normal and projectively Gorenstein Fano $3$-folds $X\subset \mathbb{P}(1^3, 2^5,3^2)$, nonsingular away from four quotient singularities $\frac{1}{2}(1,1,1)$,
and two quotient singularities $\frac{1}{3}(1,1,2)$,  with Hilbert series

\begin{center}
$P_{X}(t)=\frac{1 - 11t^4 - 8t^5 + 23t^6 + 32t^7 - 13t^8 - 48t^9 - 13t^{10} + 32t^{11} + 23t^{12} - 8t^{13} - 11t^{14} + t^{18}}{(1-t)^3(1-t^2)^5(1-t^3)^2}.$
\end{center}
\end{theorem}

\

\,

\subsection{Construction of Graded Ring Database entry with ID: 14885} \label{constr!1}
In this subsection,  we give the details of the construction for the family described in Theorem~\ref{constr!14885}.

We note that a difficult part of the arguments for this construction is the computation of singular locus of the general member of the family. As we will see below, for this part we used the computer algebra system Singular \cite{GPS01}.

Denote by $k=\mathbb{C}$ the field of complex numbers. Consider the polynomial ring $R= k[z_i, c_j]$, where $1\leq  i \leq 7$  and   $1\leq  j\leq 25$.  Let  $ R_{un}$ be the ring in Definition \ref{def!mrin}.
We substitute the variables $(c_1, \dots, c_{25})$ which appear in the definitions of the rings $R$ and $R_{un}$ with a general element of $k^{25}$ (in the sense of being outside a proper Zariski closed subset of $k^{25}$).  We denote by $\hat{R}$ the polynomial ring which occurs from $R$ and $\hat{ R}_{un}$ the ring which occurs from $ R_{un}$ after this substitution. Let $\hat{I}$ be the ideal which is obtained by the ideal $I$ and $\hat{ I}_{un}$ the ideal which obtained by the ideal $I_{un}$  after this substitution. In what follows $z_i$,  for all $i$ with $1\leq  i \leq 7$,  and $T,S,W$ are variables of degree $2$. According to this grading the ideals $\hat{I}$ and $\hat{ I}_{un}$ are homogeneous. Due to the Theorem~\ref{main!thm},   $\text{Proj} \   \hat{R}_{un}\subset \mathbb{P} (2^{10})$ is a projectively Gorenstein $3$-fold.

Let $A= k[w_{1}, w_{2}, w_{3}, z_1, z_2, z_3, z_5, T, S, W]$  be the polynomial ring over $k$ with $w_{1}, w_{2}, w_{3}$ variables of degree $1$ and the other variables of degree $2$. Consider the graded $k$-algebra homomorphism
\begin{center}
$\psi\colon \hat{R}_{un}\rightarrow A$
\end{center}
with
\begin{center}
$\psi(z_1)= z_1$, \, $\psi(z_2)= z_2$, \,   $\psi(z_3)= z_3$, \, $\psi(z_4)= f_1$,
\end{center}
\begin{center}
$\psi(z_5)= z_5$, \, $\psi(z_6)= f_2$, \, $\psi(z_7)= f_3$, \,  $\psi(T)= T$,
\end{center}
\begin{center}
$\psi(S)= S$, \, $\psi(W)= W$
\end{center}
where,

$f_1= l_1z_1+ l_2z_2+ l_3z_3+ l_4z_5+ l_5T+ l_6S+ l_7W+ l_8w_{1}^2+ l_9w_{1}w_{2}+  l_{10}w_{1}w_{3}+ l_{11}w_{2}^2
+ l_{12}w_{2}w_{3}+ l_{13}w_{3}^2$,

$f_2= l_{14}z_1+ l_{15}z_2+ l_{16}z_3+ l_{17}z_5+ l_{18}T+ l_{19}S+ l_{20}W+ l_{21}w_{1}^2+ l_{22}w_{1}w_{2}+  l_{23}w_{1}w_{3}+ l_{24}w_{2}^2
+ l_{25}w_{2}w_{3}+ l_{26}w_{3}^2$,

$f_3= l_{27}z_1+ l_{28}z_2+ l_{29}z_3+ l_{30}z_5+ l_{31}T+ l_{32}S+ l_{33}W+ l_{34}w_{1}^2+ l_{35}w_{1}w_{2}+  l_{36}w_{1}w_{3}+ l_{37}w_{2}^2
+ l_{38}w_{2}w_{3}+ l_{39}w_{3}^2$

\noindent and   $(l_1,\dots, l_{39})\in k^{39}$ are general.

Denote by $Q$ the ideal of the ring A generated by the subset   $\psi(\hat{I}_{un})$.

Let $X= V(Q)\subset \mathbb{P} (1^{3},2^{7})$. It is immediate that  $X\subset \mathbb{P}(1^3, 2^7)$ is a codimension $6$ projectively Gorenstein $3$-fold.

\begin{proposition} \label{constr1!prime}
The ring  $A/Q$ is an integral domain.
\end{proposition}

\begin{proof}
It is enough to show that the ideal $Q$ is prime. The computer algebra program Macaulay2 \cite{GS} gave us that for a specific choice of rational values for the parameters $c_i$, $l_j$, for $1\leq i\leq 25$ and $1\leq j\leq 39$ the ideal which was obtained by $Q$ is a homogeneous, codimension $6$, prime ideal with the right Betti table.
\end{proof}

In what follows, we show that the only singularities of $X\subset \mathbb{P}(1^3, 2^7)$  are $8$ quotient singularities  of type $\frac{1}{2}(1,1,1)$. According to the discussion after Definition \ref{def!fano3fol}, $X$ belongs to the Mori category. 

\begin{proposition}\label{Prop!quasismooth}
Consider   $X= V(Q)\subset \mathbb{P} (1^{3},2^{7})$. Denote by  $X_{cone}\subset \mathbb{A}^{10}$ the affine cone over $X$. 
The scheme  $X_{cone}$ is smooth outside the vertex of the cone.
\end{proposition}

\begin{proof}
We were only able to prove Proposition~\ref{Prop!quasismooth} with the help of the computer algebra program Singular \cite{GPS01}.
Our approach is similar to the approach in \cite[ Page ~ 18]{R3}.
We work over the finite field  $\mathbb{Z}/(1021)$. Differentiating the $20$ equations of  $Q$ with respect to the ten variables gives the $10\times 20$ Jacobian matrix $M^{Jac}$.
Let $J$  be the ideal of $6\times 6$ minors of $M^{Jac}$. The ideal $Q+J$ defines the singular locus of $X_{cone}$. 
Our claim is that the only singularity of the scheme  $X_{cone}$ is the vertex of the cone.
Consider the ideal  $Q+J$. Using the computer algebra program Singular we proved that $\dim ( A/(Q+J))=0$. The ideal $Q+J$ is homogeneous.  Hence, the claim is proven.
\end{proof}

\begin{proposition} \label{sing!specsing}
Consider the singular locus $Z=V(w_1,w_2,w_3)$ of the weighted projective space  $\mathbb{P}(1^3, 2^7)$. The intersection of $X$ with $Z$ consists of exactly eight points which are quotient singularities of type $\frac{1}{2}(1,1,1)$ for $X$.
\end{proposition}

\begin{proof}
Using the computer algebra program Macaulay2 \cite{GS}, we showed that the dimension of the  intersection of $X$ with $Z$ is equal to zero.  Hence, the intersection of $X$ with $Z$ is a finite set of points.
More precisely, using Macaulay2  we computed that the  intersection of $X$ with $Z$ consists of  of exactly eight points. In order to prove that the eight points  are quotient singularities of type $\frac{1}{2}(1,1,1)$ for $X$, 
we observe that the first three rows of the matrix which occurs from the jacobian matrix $M^{Jac}$ of  $Q$ by setting the variables  $w_1, w_2, w_3$  be equal to zero, are zero. Hence, due to the Proposition \ref{Prop!quasismooth}, there exists a non-zero $6\times 6$ minor in  six out of seven variables of degree $2$. In that way, we conclude that the eight points are quotient singularities  of type $\frac{1}{2}(1,1,1)$ for $X$.
\end{proof}

\begin{lemma}\label{lem!canonmod}
Let $\omega_{\hat{R}/\hat{I}}$ be the canonical module of $\hat{R}/\hat{I}$. It holds that  the canonical module $\omega_{\hat{R}/\hat{I}}$ is isomorphic to $\hat{R}/\hat{I}(-4)$.
\end{lemma}

\begin{proof}
From the minimal graded free resolution of $\hat{R}/\hat{I}$ as  $\hat{R}$-module
\[
0\rightarrow \hat{R}(-10)\rightarrow \hat{R}(-6)^{5}\rightarrow \hat{R}(-4)^{5}\rightarrow \hat{R}  
\]
and the fact that the sum of the degrees of the variables is equal to $14$ we conclude that 
\[\omega_{\hat{R}/\hat{I}}= \hat{R}/\hat{I}(10-14)= \hat{R}/\hat{I}(-4).\]
\end{proof}

\begin{proposition} \label{prop!gradres}
The minimal graded resolution of $A/Q$ as $A$-module is equal to
\begin{equation} \label{eq!resR}
0\rightarrow A(-16)\rightarrow A(-12)^{20}\rightarrow A(-10)^{64}\rightarrow A(-8)^{90}\rightarrow A(-6)^{64}\rightarrow A(-4)^{20}\rightarrow A 
\end{equation}
Moreover, the canonical module of $A/Q$ is isomorphic to  $(A/Q)(-1)$ and the Hilbert Series of  $A/Q$  as graded $A$-module is equal to 
\[
\frac{1-20t^4+64t^6-90t^8+64t^{10}-20t^{12}+t^{16}}{(1-t)^3(1-t^2)^7}.
\] 
\end{proposition}

\begin{proof}
To compute the minimal graded free resolution  of $A/Q$ we followed the method described in the proof of  \cite[ Proposition~ 3.4 ]{NP1}.
From the minimal graded free resolution (\ref{eq!resR}) of $A/Q$ and the fact that the sum of the degrees of the variables is equal to $17$ we conclude that 
\[\omega_{A/Q}= A/Q(16-17)= A/Q(-1).\]
The last conclusion of Proposition \ref{prop!gradres} follows easily from the resolution (\ref{eq!resR}). 
\end{proof}

Taking into account the Propositions \ref{Prop!quasismooth}, \ref{sing!specsing} and \ref{prop!gradres}, we conclude that $X$ is a Fano $3$-fold.

\subsection{Construction of Graded Ring Database entry with ID: 12979} \label{construct!id12979}
In this subsection, we sketch the construction for the family described in Theorem~\ref{constr!12979}.

Denote by $k=\mathbb{C}$ the field of complex numbers. Working as before,  consider the polynomial ring   $R= k[z_i, c_j]$ where $1\leq  i \leq 7$  and   $1\leq  j\leq 25$. Let  $R_{un}$ be the ring in Definition ~ \ref{def!mrin}. 
We substitute the variables $(c_2, c_3, c_4, c_6, c_7, c_8, c_{10}, c_{12}, c_{13}, \dots , c_{25})$ which appear in the definitions of the rings $R$ and $R_{un}$ with a general element of $k^{21}$ (in the sense of being outside a proper Zariski closed subset of $k^{21}$).  We denote by $\hat{R}$ the polynomial ring which occurs from $R$ and $\hat{ R}_{un}$ the ring which occurs from   R$_{un}$ after this substitution. Let $\hat{I}$ be the ideal which is obtained by the ideal $I$ and $\hat{ I}_{un}$ the ideal which obtained by the ideal $I_{un}$  after this substitution. In what follows we assume that the variables $z_1 ,c_1,  c_5, c_9,  c_{11}$ are variables of degree $1$, the variables $z_2,\dots, z_7, T$ are variables of degree $2$ and the variables $S, W$ are variables of degree $3$.
Under this grading, the ideals $\hat{I}$ and $\hat{ I}_{un}$ are homogeneous. Due to the Theorem  \ref{main!thm},   $\text{Proj} \   \hat{R}_{un}\subset \mathbb{P} (1^{5}, 2^{7}, 3^{2})$  is a projectively Gorenstein $7$-fold. 

Let $A= k[z_1, c_5, c_9,  z_2, z_3,  z_5, z_6,  T, S, W]$ be the polynomial ring with $z_1, c_5, c_9$ variables of degree $1$, $z_2, z_3,  z_5, z_6,T$ variables of degree $2$ and $S, W$ are variables of degree $3$. We consider the graded $k$-algebra homomorphism

\begin{center}
$\psi \colon  \hat{R}_{un}\rightarrow A$
\end{center}
with
\begin{center}
$\psi(z_1)= z_1$, \, $\psi(c_1)= g_1$, \, $\psi(c_5)= c_5$, \, $\psi(c_9)= c_9$,
\end{center}
\begin{center}
 $\psi(c_{11})= g_2$, \, $\psi(z_2)= z_2$, \, $\psi(z_3)= z_3$, \, $\psi(z_4)= f_1$,
\end{center}
\begin{center}
$\psi(z_5)= z_5$, \, $\psi(z_6)= z_6$, \, $\psi(z_7)= f_2$, \, $\psi(T)= T$,
\end{center}
\begin{center}
$\psi(S)= S$, \, $\psi(W)= W$
\end{center}
where,

$g_1= l_1 z_1+ l_2 c_5+ l_3 c_9$,

$g_2= l_4 z_1+ l_5c_5+ l_6c_9$,

$f_1= l_7z_2+ l_8z_3+ l_9z_5+ l_{10}z_6+  l_{11} T+  l_{12}z_1^2+  l_{13}z_1c_5+ l_{14}z_1c_9+ l_{15}c_5^2+ l_{16}c_5c_9+ l_{17}c_9^2$,

$f_2= l_{18}z_2+ l_{19}z_3+ l_{20}z_5+ l_{21}z_6+  l_{22}T+  l_{23}z_1^2+  l_{24}z_1c_5+ l_{25}z_1c_9+ l_{26} c_5^2+ l_{27}c_5c_9+ l_{28}c_9^2$

\noindent and   $(l_1,\dots, l_{28})\in k^{28}$ are general.

Denote by $Q$ the ideal of the ring $A$ generated by the subset   $\psi(\hat{I}_{un})$.

Let $X= V(Q)\subset \mathbb{P} (1^{3}, 2^{5},3^{2})$.  It is immediate that  $X\subset \mathbb{P} (1^{3}, 2^{5},3^{2})$ is a codimension $6$ projectively Gorenstein $3$-fold. 

Repeating the arguments which were used for the construction which was described in Subsection \ref{constr!1} we proved that   $X\subset \mathbb{P}(1^3, 2^5,3^2)$ is a Gorenstein Fano $3$-fold nonsingular away from four quotient singularities  $\frac{1}{2}(1,1,1)$ and two quotient singularities  $ \frac{1}{3}(1,1,2)$.

\section* {Acknowledgements} 

\label{sec!acknowledgements}
I would like to express my deep gratitude to my supervisor Stavros Papadakis for his valuable guidance and ideas which contributed immensely in the present work. As well as for the important suggestions which have improved the presentation of the paper. I benefited from experiments with the computer algebra program Macaulay2~\cite{GS}. This work is part of a Univ. of Ioannina Ph.D. thesis, financially supported by the Special Account for Research Funding (E.L.K.E) of University of Ioannina (UOI) under the program with code $82561$ and title 
\q{Program of financial support for Ph.D. students and postdoctoral researchers}.

\end{document}